\documentclass[a4paper,11pt]{amsart}

\newfont{\cyr}{wncyr10 scaled 1100}

\usepackage[usenames,dvipsnames]{color}

\usepackage[left=2.7cm,right=2.7cm,top=3.5cm,bottom=3cm]{geometry}
\usepackage{amsthm,amssymb,amsmath,amsfonts,mathrsfs,amscd,yfonts}
\usepackage{turnstile}
\usepackage[latin1]{inputenc}
\usepackage[all]{xy}
\usepackage{latexsym}
\usepackage{longtable}

\theoremstyle{plain}
\newtheorem{theorem}{Theorem}[section]

\newtheorem{lemma}[theorem]{Lemma}

\newtheorem{propo}[theorem]{Proposition}
\newtheorem{coro}[theorem]{Corollary}
\newtheorem{conj}[theorem]{Conjecture}

\theoremstyle{definition}

\newtheorem{question}[theorem]{Question}
\newtheorem{examplewr}[theorem]{Example}

\theoremstyle{remark}
\newtheorem{obswr}[theorem]{Observation}
\newtheorem{remarkwr}[theorem]{Remark}

\newenvironment{remark}{\begin{remarkwr}\begin{upshape}}{\end{upshape}\end{remarkwr}}

\DeclareMathOperator{\BK}{BK}

\DeclareMathOperator{\pr}{pr}
\DeclareMathOperator{\cyc}{cyc}

\DeclareMathOperator{\Spf}{Spf}

\DeclareMathOperator{\Eis}{Eis}

\DeclareMathOperator{\ad}{ad}

\DeclareMathOperator{\crit}{crit}

\newcommand{\Lp}{{\mathscr{L}_p}}

\newcommand{\cW}{\mathcal W}

\newcommand{\Q}{\mathbb{Q}}
\newcommand{\Z}{\mathbb{Z}}

\newcommand{\Gal}{\mathrm{Gal\,}}
\newcommand{\GL}{\mathrm{GL}}

\newcommand{\Frob}{\mathrm{Fr}}

\newcommand{\Fr}{\mathrm{Fr}}

\newcommand{\ord}{{\mathrm{ord}}}
\newfont{\gotip}{eufb10 at 12pt}

\newcommand{\cO}{{\mathcal O}}

\newcommand{\hg}{{\mathbf{g}}}
\newcommand{\hh}{{\mathbf h}}

\DeclareMathOperator{\Hom}{Hom}

\include{thebibliography}

\begin{document}

\title[Cyclotomic derivatives and a Gross--Stark formula]{Cyclotomic derivatives of Beilinson--Flach classes and a new proof of a Gross--Stark formula}

\author{\'Oscar Rivero}

\begin{abstract}
We give a new proof of a conjecture of Darmon, Lauder and Rotger regarding the computation of the $\mathcal L$-invariant of the adjoint of a weight one modular form in terms of units and $p$-units. While in our previous work with Rotger the essential ingredient was the use of Galois deformations techniques following the computations of Bella\"iche and Dimitrov, we propose a new approach exclusively using the properties of Beilinson--Flach classes. One of the key ingredients is the computation of a cyclotomic derivative of a cohomology class in the framework of Perrin-Riou theory, which can be seen as a counterpart to the earlier work of Loeffler, Venjakob and Zerbes. We hope that this method could be applied to other scenarios regarding exceptional zeros, and illustrate how this could lead to a better understanding of this setting by conjecturally introducing a new $p$-adic $L$-function whose special values involve information just about the unit of the adjoint (and not also the $p$-unit), in the spirit of the conjectures of Harris and Venkatesh.
\end{abstract}

\address{O. R.: Departament de Matem\`{a}tiques, Universitat Polit\`{e}cnica de Catalunya, C. Jordi Girona 1-3, 08034 Barcelona, Spain}
\email{oscar.rivero@upc.edu}


\maketitle

\tableofcontents

\section{Introduction}

In our series of works with Rotger \cite{RiRo1}, \cite{RiRo2}, we proposed a systematic study of the conjecture of Darmon, Lauder and Rotger \cite{DLR} on $p$-adic iterated integrals in terms of certain cohomology classes constructed from the $p$-adic interpolation of Beilinson--Flach elements. This conjecture may be subsumed in a broader programme comprising both the Gross--Stark conjectures and also the celebrated {\it Elliptic Stark Conjectures}, which shed some light on the arithmetic of elliptic curves of rank 2.

Let $\chi$ be a Dirichlet character of level $N \geq 1$, and let $S_1(N,\chi)$ stand for the space of cuspidal modular forms of weight 1, level $N$ and nebentype $\chi$. Let $g = \sum_{n \geq 1} a_n q^n \in S_1(N,\chi)$ be a normalized newform and let $g^* = g \otimes \bar \chi$ denote its twist by the inverse of its nebentype. Let
\[ \varrho_g: \Gal(H_g/\Q) \hookrightarrow \GL(V_g) \simeq \GL_2(L), \quad \varrho_{\mathrm{ad}^0(g)}: \Gal(H/\Q) \hookrightarrow \GL(\mathrm{ad}^0(g)) \simeq \GL_3(L) \] denote the Artin representations associated to $g$ and its adjoint, respectively. Here $H_g \supseteq H$ denote the finite Galois extensions of $\mathbb Q$ cut out by these representations, and $L$ is a sufficiently large finite extension of $\mathbb Q$ containing their traces.

Label and order the roots of the $p$-th Hecke polynomial of $g$ as $X^2-a_p(g)X+\chi(p) = (X-\alpha)(X-\beta)$. We assume throughout that
\begin{enumerate}
\item[(H1)] The reduction of $\varrho_g$ mod $p$ is irreducible;

\item[(H2)] $g$ is $p$-distinguished, i.e.\,$\alpha \ne \beta \, (\mathrm{mod} \, p)$, and

\item[(H3)] $\varrho_g$ is not induced from a character of a real quadratic field in which $p$ splits.
\end{enumerate}

Hence, following \cite{RiRo2}, there are {\em four} (a priori distinct!) cohomology classes
\begin{equation}\label{clases}
\kappa(g_{\alpha},g_{1/\alpha}^*), \quad \kappa(g_{\beta},g_{1/\beta}^*), \quad \kappa(g_{\alpha},g_{1/\beta}^*), \quad \kappa(g_{\beta},g_{1/\alpha}^*) \in H^1(\mathbb Q, \ad^0(g)(1)).
\end{equation}
However, and as it was proved in \cite[Sections 3.2, 3.4]{RiRo1}, it happens that \begin{equation}\label{zeros}
\kappa(g_{\alpha},g_{1/\beta}^*) = \kappa(g_{\beta},g_{1/\alpha}^*) = 0
\end{equation}
due to an exceptional zero phenomenon which can be explained by the vanishing of an Euler factor.
To overcome this situation, in Section 3.3 of loc.\,cit. we had constructed certain derivatives of those classes, but it turns out that the definition we had used there is not completely useful for our purposes. Roughly speaking, we had taken the derivative along one of the weight directions associated to the Hida family interpolating one of the modular forms, while towards obtaining a more flexible and arithmetically interesting setting we need to consider also the {\it cyclotomic} derivative. This is an analogue situation to the scenario of \cite{Buy2} and \cite{Ven}, where the computation of the derivative of the Mazur--Kitagawa $p$-adic $L$-function along a certain direction of the weight space was relatively easy using the classical theory of Heegner points (and had already been carried out by Bertolini and Darmon \cite{BD}), but the computation of the {\it cyclotomic} derivative required new ideas. Hence, this work may be thought as a counterpart to the approach of B\"uy\"ukboduk and Venerucci to the exceptional zero phenomenon, but in the easier case where elliptic curves are replaced by unit groups (and hence one can circumvent the technical complications introduced by the use of Nekovar's height theory). Similar results had been obtained by Loeffler, Venjakob and Zerbes \cite{LVZ}, and one can see our computations as the {\it dual} of Proposition 2.5.5 and Theorem 3.1.2 of loc.\,cit. We refer also to the seminal works of Benois \cite{Ben}, \cite{Ben2} where similar questions are addressed.

Our main result in \cite{RiRo1} was the computation of a special value formula for the Hida--Rankin $p$-adic $L$-function at weight one (alternatively, the derivative of the adjoint of the modular form). This is specially intriguing since that function, that we denote as $L_p(g,g^*,s)$, cannot be directly defined in terms of an interpolation property, and requires to consider the $p$-adic variation of the modular forms $(g,g^*)$ along a Hida family. Indeed, it depends on the choice of a $p$-stabilization for $g$. We sometimes write $L_p^{g_{\alpha}}(g,g^*,s)$ to emphasize this dependence. In \cite[Section 1]{DLR} it is shown that
\[
\dim_L (\cO_H^\times \otimes \mathrm{ad}^0(g))^{G_\Q} = 1, \quad \dim_L (\cO_H[1/p]^\times/p^{\mathbb Z} \otimes \mathrm{ad}^0(g))^{G_\Q} = 2.
\]
Fix a generator $u$ of $ (\cO_H^\times \otimes \mathrm{ad}^0(g))^{G_\Q}$ and also an element $v$ of $(\cO_H^{\times}[1/p]^{\times} \otimes \mathrm{ad}^0(g))^{G_\Q}$ in such a way that $\{ u, v\}$ is a basis of $(\cO_H[1/p]^\times/p^\Z \otimes \mathrm{ad}^0(g))^{G_\Q}$. The element $v$ may be chosen to have $p$-adic valuation $\ord_p (v)=1$, and we do so.

Viewed as a $G_{\Q_p}$-module, $\mathrm{ad}^0(g)$ decomposes as $\mathrm{ad}^0(g) = L \oplus L^{\alpha \otimes \bar \beta} \oplus L^{\beta \otimes \bar \alpha}$, where each line is characterized by the property that the arithmetic Frobenius $\Frob_p$ acts on it with eigenvalue $1$, $\alpha/\beta$ and $\beta/\alpha$, respectively. Let $H_p$ denote the completion of $H$ in $\bar\Q_p$  and let
\[
u_1, \,\, u_{\alpha \otimes \bar \beta},  \,\,  u_{\beta \otimes \bar \alpha},  \,\,  v_1,  \,\,  v_{\alpha \otimes \bar \beta},  \,\,  v_{\beta \otimes \bar \alpha} \in H_p^\times \otimes_{\Q} L \quad (\mathrm{mod} \, L^\times)
\]
denote the projection of the elements $u$ and $v$ in $(H_p^\times \otimes \mathrm{ad}^0(g))^{G_{\Q_p}}$ to the above lines.

Then, we have the following:
\begin{theorem}\label{central}
Assume that hypothesis (H1)-(H3) hold. Then, the following equality holds up to multiplication by a scalar in $L^{\times}$ \[ L_p^{g_{\alpha}}(g,g^*,1) = \frac{\log_p(u_{\alpha \otimes \bar \beta}) \log_p(v_1) - \log_p(v_{\alpha \otimes \bar \beta}) \log_p(u_1)}{\log_p(u_{\alpha \otimes \bar \beta})}. \]
\end{theorem}

The proof we had given in \cite[Section 4]{RiRo1} was lengthy and made use of the results of Bellaiche--Dimitrov computing the tangent space of a deformation problem, together with some elements taken from the earlier work \cite{DLR2}. In a certain way, that proof mimicked the approach of Greenberg--Stevens \cite{GS} to the exceptional zero phenomenon for elliptic curves with split multiplicative reduction. However, the authors had observed a tantalizing connection with the theory of Beilinson--Flach elements, that were affected by a similar exceptional zero phenomenon. This allowed us to interpret {\it derived} classes of Beilinson--Flach elements in terms of the units $\{u,v\}$, but does not give any new insight into the proof of Theorem \ref{central}, as we had initially expected. This work may be seen as a culmination of the purpose that the authors had when they began to write both \cite{RiRo1} and \cite{RiRo2}, that was proving the Gross--Stark conjecture of Darmon, Lauder and Rotger using just the properties of Beilinson--Flach elements and the flexibility provided by the notion of {\it derivatives}. This is just another instance of the power of Euler systems when dealing with arithmetic questions.

We can give, with these ideas at hand, a different proof of the main theorem of \cite{RiRo1}. This can be seen as the counterpart to the approach of Kobayashi \cite{Ko} to the Mazur--Tate--Teitelbaum conjecture in rank 0, since he reproves the result of Greenberg and Stevens \cite{GS} using the properties of Kato's cohomology classes.

Our proof is a combination of four main ideas (together with the same starting point coming from Hida's theory of improved $p$-adic $L$-functions):
\begin{enumerate}
\item[(0)] The results of Hida \cite{Hi1}, \cite{Hi2}, which reduce the conjecture to the computation of the derivative of the Frobenius eigenvalue along the weight direction. This part is common to our earlier work \cite{RiRo1}.
\item[(1)] The local properties at $p$ of Beilinson--Flach elements, which give an expression, {\it up to multiplication by a $p$-adic scalar}, for the derived class $\kappa'(g_{\alpha},g_{1/\beta}^*)$ in terms of the units $u$ and $v$, where here the derivative is taken along any arbitrary direction of the weight space.
\item[(2)] A comparison between the different reciprocity laws and the observation that knowing two {\it weight} derivatives, together with the vanishing of the class $\kappa(\hg,\hg^*)$ along the line $(\ell,\ell,\ell-1)$, allows us to determine the cyclotomic derivative of the class.
\item[(3)] An explicit reciprocity law for the $\Lambda$-adic class $\kappa(\hg,\hg^*)$, obtained when $g$ and $g^*$ vary over Hida families $\hg$ and $\hg^*$, respectively. This was first proved by Kings--Loeffler--Zerbes \cite{KLZ}. In our situation, there is an exceptional vanishing, and hence we may consider a {\it derived} reciprocity law, in the sense of \cite{RiRo1}. This gives an expression for the {\it weight derivative} of the Beilinson--Flach class in terms of an unknown $p$-adic period and involving also the $\mathcal L$-invariant of the adjoint of $g_{\alpha}$.
\item[(4)] The results of B\"uy\"ukboduk \cite{Buy1}, \cite{Buy2} and Venerucci \cite{Ven} around Coleman maps, which allow us to relate the {\it cyclotomic derivative} of the Beilinson--Flach class to the Hida--Rankin $p$-adic $L$-function. This part can be also recast, by duality, in terms of the computations developed in \cite{LVZ}. Comparing this result with (3), we get a formula for the $\mathcal L$-invariant, and consequently for the special $p$-adic $L$-value.
\end{enumerate}

Observe that the study of universal norms has also allowed Roset, Rotger and Vatsal \cite{RRV} to reinterpret the $\mathcal L$-invariant of Theorem \ref{central} in terms of an algebraic avatar initially defined by Greenberg \cite{Gr}.

However, Theorem \ref{central} is not completely satisfactory towards the understanding of the arithmetic of the adjoint of a weight one modular form, since it involves both the unit and the $p$-unit attached to the Galois representation. It is natural to expect a putative refinement of the previous result in the spirit of the conjectures of Harris--Venkatesh \cite{HV}, with a $p$-adic $L$-function whose special values encode information just about the unit $u$. The last section of this note is devoted to discuss the following conjecture in the framework provided by Perrin-Riou maps.

\begin{conj}\label{deseo}
There exists an analytic $p$-adic $L$-function $L_p^{\Eis}(g,g^*,s)$, appropriately related with the cohomology class $\kappa(g_{\alpha},g_{1/\alpha}^*)$ via the theory of Perrin-Riou maps, and such that \[ L_p^{\Eis}(g,g^*,1) = \log_p(u_1) \pmod{L^{\times}}. \]
\end{conj}

We finish this preliminary discussion by pointing out that the formalism we discuss regarding exceptional zeros and derivatives at the level of Euler systems suggests the possibility of developing an axiomatic treatment mimicking the general conjectures of Greenberg and Benois. We believe that this could be specially intriguing in the scenario of diagonal cycles or in its recent generalizations to the Siegel case, where a cohomology class may simultaneously encode information about the behavior of different, and a priori unrelated, $p$-adic $L$-functions.

\vskip 12pt

The organization of this note is as follows. Section 2 discusses the motivational case of circular units, where these same phenomena arise and that can serve as a motivation for our later work. Section 3 recalls the notations and results of \cite{RiRo1} around Beilinson--Flach elements which are needed in the proof. Section 4 contains the main results of the article and discuss the new proof of Theorem \ref{central} using the notion of {\it derived} Belinson--Flach elements. Section 5 proposes an alternative interpretation of the previous results in terms of deformation theory. Finally, Section 6 discusses the $p$-adic $L$-function of Conjecture \ref{deseo} and its relationship with the periods of weight one modular forms.

\vskip 12pt

{\bf Acknowledgements.} This research began during a research stay of the author at McGill University in Fall 2019. It is a pleasure to sincerely thank Henri Darmon, who organised a small seminar around the symmetric square $p$-adic $L$-function where the author could explain his previous work with Rotger, and was encouraged to try to find an alternative proof of those results in terms of derived cohomology classes. I am indebted to Denis Benois, for many insights and useful conversations around this work. Finally, I thank Victor Rotger for many stimulating conversations around the topic of this note, which is the natural continuation of our earlier works on this question. This project has received funding from the European Research Council (ERC) under the European Union's Horizon 2020 research and innovation programme (grant agreement No. 682152). The author has also received financial support through ``la Caixa" Banking Foundation (grant LCF/BQ/ES17/11600010).

\section{Analogy with the case of circular units}

The situation we want to deal with has a clear parallelism with the case of circular units, that we now recall. We fix a Dirichlet character $\chi$, and write $H$ for the field cut out by the Artin representation attached to it, and $L$ for its coefficient field. The case where $\chi$ is odd gives rise to an exceptional vanishing of the Deligne--Ribet $p$-adic $L$-function $L_p(\chi \omega,s)$ at $s=0$ when $\chi(p)=1$, where $\omega$ is the Teichm\"uller character. Under the assumption that $\chi$ is odd, $(\mathcal O_H^{\times} \otimes L)^{\chi}$ is a zero-dimensional $L$-vector space, while $(\mathcal O_H[1/p]^{\times} \otimes L)^{\chi}$ has dimension 1 if $\chi(p)=1$. In this case, choosing a non-zero element $v_{\chi}$ of the latter space, we may define an $\mathcal L$-invariant
\begin{equation}
\mathcal L(\chi) = -\frac{\log_p(v_{\chi})}{\ord_p(v_{\chi})},
\end{equation}
where we are implicitly choosing a prime of $H$ above $p$. Then,
\begin{equation}
L_p'(\chi \omega, s) = \mathcal L(\chi) \cdot L(\chi,0).
\end{equation}
See \cite{DDP} for more details on that and for a broader treatment in the setting of totally real number fields.

In the case where $\chi$ is even, the situation is ostensibly different. Here, $(\mathcal O_H^{\times} \otimes L)^{\chi}$ is one-dimensional and we may fix a generator $c_{\chi}$ of it, that we call the {\it circular} unit associated to $\chi$. We take it, as usual, as a weighted combination of cyclotomic units \[ c_{\chi} = \prod_{a=1}^{N-1} (1-\zeta_N^a)^{\chi^{-1}(a)}, \] where $\zeta_N$ is a fixed primitive $N$-th root of unity and the notation $(1-\zeta_N^a)^{\chi^{-1}(a)}$ means $(1-\zeta_N^a) \otimes \chi^{-1}(a)$. Moreover, if we further assume that $\chi(p)=1$, $(\mathcal O_H[1/p]^{\times} \otimes L)^{\chi}$ has dimension 2, and we may consider a basis of the form $\{c_{\chi},v_{\chi}\}$, with the convention that $\ord_p(v_{\chi})=1$.

Given any even, non trivial Dirichlet character, one always has Leopoldt's formula, which asserts that
\begin{equation}\label{kl}
L_p(\chi,1) = -\frac{(1-\chi(p)p^{-1})}{\mathfrak g(\bar \chi)} \cdot \log_p(c_{\chi}).
\end{equation}

We may understand the previous result in the more general setting of reciprocity laws. For that purpose, let $\Lambda = \mathbb Z_p[[\mathbb Z_p^{\times}]]$. By adding $p$-power conductors to $\chi$ and considering a family of coherent units along the cyclotomic tower, one may construct a $\Lambda$-adic class $\kappa(\chi,s)$, whose bottom layer vanishes when $\chi(p)=1$. We keep this assumption on the character $\chi$ all along this section and write \[ \kappa(\chi,s) \in H^1(\mathbb Q,L_p(\bar \chi) \otimes \Lambda(\varepsilon_{\cyc} \underline{\varepsilon}_{\cyc})) \] for the $\Lambda$-adic class, where $\varepsilon_{\cyc}$ is the usual cyclotomic character and $\underline{\varepsilon}_{\cyc}$ stands for the $\Lambda$-adic cyclotomic character. Given $s \in \mathbb Z$, let $\nu_s: \Lambda(\underline{\varepsilon}_{\cyc}) \rightarrow \mathbb Z_p$ be the ring homomorphism sending the group-like element $a \in \mathbb Z_p^{\times}$ to $a^s$. This induces a $G_{\mathbb Q}$-equivariant specialization map \[ \nu_s: \, \Lambda(\underline{\varepsilon}_{\cyc}) \rightarrow \mathbb Z_p(s) \] and gives rise to a collection of global cohomology classes \[ \kappa(\chi,s) \in H^1(\mathbb Q, L_p(\bar \chi)(s)). \]

The Perrin-Riou formalism allows us to understand $L_p(\chi,s)$ as the image under a Coleman map (also named as Perrin-Riou map, or Perrin-Riou regulator) of the local class $\kappa_p(\chi,s)$ \[ \mathcal L_{\chi} \, : \, H^1(\mathbb Q_p, L_p(\bar \chi) \otimes \Lambda(\varepsilon_{\cyc} \underline{\varepsilon}_{\cyc})) \longrightarrow I^{-1} \Lambda, \quad \mathcal L_{\chi}(\kappa_p(\chi,s)) = L_p(\chi,s), \] where $I$ is the ideal of $\Lambda$ corresponding to the specialization at $s=1$ (see \cite[Section 8.2]{KLZ} for the precise definition). This map interpolates either the dual exponential map (for $s \leq 0$) or the Bloch--Kato logarithm (for $s \geq 1$). Unfortunately, the bottom layer $\kappa(\chi,1)$ vanishes when $\chi(p)=1$. Following the construction of \cite[Section 3]{Buy1}, there is a derived class $\kappa'(\chi,s)$, defined as the unique class satisfying that
\begin{equation} \kappa(\chi,s) = \frac{\gamma-1}{\log_p(\gamma)} \cdot \kappa'(\chi,s),
\end{equation}
where $\gamma$ is a fixed topological generator of $\mathbb Z_p[[1+p\mathbb Z_p]]$. It is also proved in \cite{Buy1} that $\kappa'(\chi,1)$ belongs to an extended Selmer group, which in this case may be identified with the group of $p$-units where the Galois group acts via $\chi$ (we insist that when $\chi$ is even this space is two-dimensional); this is coherent with the fact that taking Iwasawa derivatives along a $\mathbb Z_p$-extension can only introduce ramification at $p$. Hence, the exceptional zero phenomenon does not appear at the level of $p$-adic $L$-functions, since at least generically $L_p(\chi,1) \neq 0$, but at the level of cohomology classes.

Moreover, in the cases where $\chi(p)=1$ one can also define an {\it improved} map  \[ \widetilde{\mathcal L_{\chi}} = \frac{\gamma-1}{\frac{1}{p} \log_p(\gamma)} \times \mathcal L_{\chi} : \, H^1(\mathbb Q_p, L_p(\bar \chi)(\varepsilon_{\cyc} \underline{\varepsilon}_{\cyc})) \longrightarrow I^{-1} \Lambda, \] where $\varepsilon_{\cyc}$ is the usual cyclotomic character. Therefore, \[  \widetilde{\mathcal L_{\chi}}(\kappa_p'(\chi,s)) = p \cdot L_p(\chi,s). \]

The computations done in \cite[Section 6.2]{Buy1}, in particular Remark 6.5, show that the map $\widetilde{\mathcal L_{\chi}}$, when specialized at $s=1$, is given by the order map (applied in this case to the derived class). The key point is a computation of the universal norms over the cyclotomic tower, as well as the use of Lemma 6.4 of loc.\,cit (see also \cite[Section 3]{Ven}).  Hence, we have the following (identifying as usual the cohomology classes with the corresponding units via the standard Kummer map).

\begin{propo}
The element $\kappa'(\chi,1) \in (\mathcal O_H[1/p]^{\times} \otimes L)^{\chi}$ satisfies that \[ L_p(\chi,1) = -\frac{1-p^{-1}}{\mathfrak g(\bar \chi)} \cdot \ord_p(\kappa'(\chi,1)). \]
\end{propo}
\begin{proof}
This follows after combining the results of \cite[Section 6.2]{Buy1} on the properties of the map $\tilde{\mathcal L}_{\chi}$ with Solomon's computations, showing that the $p$-adic valuation of the derived class (sometimes referred as the {\it wild} cyclotomic unit) agrees with the logarithm of the circular unit (see also Proposition 4.1 of loc.\,cit.).
\end{proof}

\section{Beilinson--Flach elements}

\subsection{The three variable cohomology classes}

Let $\hg \in \Lambda_{\hg}[[q]]$ and $\hg^* \in \Lambda_{\hg}[[q]]$ be two Hida families of tame conductor $N$ and tame nebentype $\chi$ and $\bar \chi$, where $\Lambda_{\hg}$ is a finite flat extension of the Iwasawa algebra $\Lambda = \mathbb Z_p[[\mathbb Z_p^{\times}]]$. Let $\Lambda_{\hg \hg^*} = \Lambda_{\hg} \hat \otimes \Lambda_{\hg} \hat \otimes \Lambda$, and consider also the $\Lambda_{\hg}$-modules afforded by the Hida families attached to $\hg$ and $\hg^*$, that we denote $\mathbb V_{\hg}$ and $\mathbb V_{\hg^*}$, respectively. Finally, consider the $\Lambda_{{\bf gg^*}}$-module
\begin{equation}\label{Vgh}
\mathbb V_{{\bf gg^*}}:=\mathbb V_{{\bf g}} \hat \otimes_{\mathbb Z_p} \mathbb V_{{\bf g^*}} \hat \otimes_{\mathbb Z_p} \Lambda(\varepsilon_{\cyc} \underline{\varepsilon}_{\cyc}^{-1}),
\end{equation}
where $\Lambda(\underline{\varepsilon}_{\cyc}^{-1})$ stands for the twist of $\Lambda$ by the inverse of the $\Lambda$-adic cyclotomic character. The formal spectrum of $\Lambda_{\hg \hg^*}$ is endowed with certain distinguished points, the so-called crystalline points, denoted as $\cW_{{\bf gh}}^\circ$ and indexed by triples $(y,z,\sigma)$; we refer the reader to Section 2 of loc.\,cit. for the definitions.

The $\Lambda$-adic Galois representation $\mathbb V_{\hg \hg^*}$ is characterized by the property that for $(y,z,\sigma) \in \cW_{{\bf gg^*}}^\circ$, with $\sigma$ of weight $s\in \Z$, \eqref{Vgh} specializes to
$$
\mathbb V_{\hg \hg^*}(y,z,\sigma) = V_{g_y} \otimes V_{g_z^*}(1-s),
$$
the $(1-s)$-th Tate twist of the tensor product of the Galois representations attached to $g_y$ and $g_z^*$.

Fix $c \in \Z_{>1}$ such that $(c,6pN)=1$. \cite[Theorem A]{KLZ} yields a three-variable $\Lambda$-adic global Galois cohomology class
$$
\kappa^c(\hg,\hg^*) \in H^1(\mathbb Q, \mathbb V_{\hg \hg^*})
$$
that is referred to as the Euler system of Beilinson--Flach elements associated to $\hg$ and $\hg^*$. We denote by $\kappa_p^c(\hg,\hg^*) \in H^1(\mathbb Q_p, \mathbb V_{\hg \hg^*})$ the image of $\kappa^c(\hg,\hg^*)$ under the restriction map. Since $c$ is fixed throughout, we may sometimes drop it from the notation. This constant does make an appearance in fudge factors accounting for the interpolation properties satisfied by the Euler system, but in all cases we are interested in these fudge factors do not vanish and hence do not pose any problem for our purposes.

Given a crystalline arithmetic point $(y,z,s) \in \cW_{\hg \hg^*}^\circ$ of weights $(\ell,m,s)$, set for notational simplicity throughout this section $g=g_y^\circ$, $g^*=(g_z^*)^\circ$. With these notations, $g_y$ (resp.\,$g_z^*$) is the $p$-stabilization of $g$ (resp.\,$g^*$) with $U_p$-eigenvalue $\alpha_g$ (resp.\,$\alpha_{g^*}$).

Define
\begin{equation}\label{home}
\kappa(g_y,g_z^*,s):=\kappa({\bf g}, {\bf g}^*)(y,z,s) \in H^1(\mathbb Q, V_{g_y} \otimes V_{g_z^*}(1-s))
\end{equation}
as the specialisation of $\kappa({\bf g},{\bf g}^*)$ at $(y,z,s)$.

As explained in \cite[Section 2]{DR2.5}, the spaces $\mathbb V_{\hg}$ and $\mathbb V_{\hg^*}$, as $G_{\mathbb Q_p}$-modules, are endowed with a stable filtration \[ 0 \longrightarrow \mathbb V_{\hg}^+ \longrightarrow \mathbb V_{\hg} \longrightarrow \mathbb V_{\hg}^- \longrightarrow 0, \] where $\mathbb V_{\hg}^+$ and $\mathbb V_{\hg}^-$ are flat $\Lambda_{\hg}$-modules with a $G_{\mathbb Q_p}$-action, locally free of rank one over $\Lambda_{\hg}$, and such that the quotient $\mathbb V_{\hg}^-$ is unramified. Define the $G_{\mathbb Q_p}$-subquotient $\mathbb V_{\hg \hg^*}^{-+} := \mathbb V_{\hg}^- \hat \otimes \mathbb V_{\hg^*}^+$ of $\mathbb V_{\hg \hg^*}$, which is of rank one over the two-variable Iwasawa algebra $\Lambda_{\hg} \hat \otimes \Lambda_{\hg^*}$ (this quotient makes sense because of \cite[Proposition 8.1.7]{KLZ}).

Then, one may consider a Perrin Riou map
\begin{equation}
\langle \mathcal L_{\hg \hg^*}^{-+}, \eta_{\hg} \otimes \omega_{\hg^*} \rangle : H^1(\mathbb Q_p, \mathbb V_{\hg \hg^*}^{-+} \hat \otimes \Lambda(\varepsilon_{\cyc} \underline{\varepsilon}_{\cyc}^{-1})) \longrightarrow \Lambda_{\hg \hg^*} \otimes \mathbb Q_p(\mu_N).
\end{equation}
This application satisfies an explicit reciprocity law, which is the content of \cite[Theorem B]{KLZ}, and which asserts that
\begin{equation}
\langle \mathcal L_{\hg \hg^*}^{-+}(\kappa_p^{-+}(\hg,\hg^*)), \eta_{\hg} \otimes \omega_{\hg^*} \rangle = \mathcal A(\hg,\hg^*) \cdot L_p(\hg,\hg^*),
\end{equation}
where $\mathcal A(\hg,\hg^*)$ is the Iwasawa function of \cite[Theorem 10.2.2]{KLZ} and $\kappa_p^{-+}(\hg,\hg^*)$ stands for the composition of the localization-at-$p$ map with the projection $\mathbb V_{\hg \hg^*} \rightarrow \mathbb V_{\hg \hg^*}^{-+}$ in local cohomology.

The different specializations of the map $\langle \mathcal L_{\hg \hg^*}^{-+}, \eta_{\hg} \otimes \omega_{\hg^*} \rangle$ can be expressed in terms of the Bloch--Kato logarithm or the dual exponential map. In particular, we are interested in the specializations of the class $\kappa(\hg,\hg^*)$ at weights $(1,1,0)$, and more generally, weights $(\ell,\ell,\ell-1)$, where the Perrin-Riou map interpolates, up to some explicit Euler factors, the Bloch--Kato logarithm. Unfortunately, these factors may vanish in the self-dual case, and one must recast to the concept of derivatives of Euler systems.

\subsection{Derivatives of Beilinson--Flach elements}

We keep the notations fixed in the introduction regarding weight one modular forms and units for the adjoint representation. Further, we fix a point of weight one $y_0 \in \Spf(\Lambda_{\hg})$ such that $\hg_{y_0} = g_{\alpha}$ and $\hg^*_{y_0} = g_{1/\beta}^*$.

In \cite[\S3]{RiRo1}, we proved that the class $\kappa(\hg,\hg^*)$ was zero over the line corresponding to the Zariski closure of points of weights $(\ell,\ell,\ell-1)$ due to the vanishing of an Euler factor in the interpolation property for the Eisenstein class. In loc.\,cit. we also constructed a derivative along the $y$-direction (alternatively, keeping $y$ fixed and varying at a time the other two variables). However, since the weight space is three-dimensional, it makes sense to ask about the derivative along any other direction.

Since along the line $(\ell,\ell,\ell-1)$ the class is identically zero, the derivative also vanishes. Hence, by an elementary argument in linear algebra, it suffices to determine the derivative along any other two {\it independent directions} to capture all the {\it first-order} information.

\begin{remark}
Observe that we are using the results of \cite{LZ} which assert that the Beilinson--Flach elements lie in the part corresponding to the adjoint in the decomposition $V_{gg^*} = \ad^0(V_g) \oplus 1$. Alternatively, one can consider the projection to the alternate part for weight greater than one and then apply a limit argument. Moreover, and following the discussion of \cite[Section 4]{RiRo2}, one has that the projection of the subspace $p^{\mathbb Z}$ to the adjoint component is trivial.
\end{remark}

For the sake of simplicity, and since this suffices for our purposes, we restrict to the local classes at $p$. We closely follow the ideas of \cite[Section 3]{RiRo1}, showing that the specialization at weight 1 of the $\Lambda$-adic class is a linear combination of the unit $u$ and the $p$-unit $v$, as described in the introduction.

\begin{lemma}
The derivative of $\kappa_p(\hg,\hg^*)$ at $(y_0,y_0,0)$ along the $y$-direction (keeping fixed both $z$ and $s$) satisfies the following equality in $H^1(\mathbb Q_p, V_{gg^*}(1))$
\begin{equation}\label{wt-y}
\frac{\partial \kappa_p(\hg,\hg^*)}{\partial y} \Big|_{(y_0,y_0,0)} = \Omega \cdot (\log_p(v_{\alpha \otimes \bar \beta}) u - \log_p(u_{\alpha \otimes \bar \beta}) v),
\end{equation}
where $\Omega \in H_p$ and we have made use of the usual notations for writing directional derivatives.
\end{lemma}
\begin{proof}
According to the properties of the cohomology classes established in \cite[Section 3]{RiRo1}, the left hand side may be written as a combination of the units $u$ and $v$. Then, the result follows by applying \cite[Proposition 8.1.7]{KLZ} to $\kappa({\hg},{\hg}^*)$ in order to conclude that its projection to $\mathbb V_{\hg \hg^*}^{--}$ is identically zero, and therefore the same is true for its derivative. Specializing at $(y_0,y_0,0)$, the result automatically follows.
\end{proof}

\begin{lemma}
The derivative of $\kappa_p(\hg,\hg^*)$ at $(y_0,y_0,0)$ along the $z$-direction (keeping fixed both $y$ and $s$) satisfies the following equality in $H^1(\mathbb Q_p, V_{gg^*}(1))$ up to a factor in $L^{\times}$
\begin{equation}\label{wt-z}
\frac{\partial \kappa_p(\hg,\hg^*)}{\partial z} \Big|_{(y_0,y_0,0)} = \Omega \cdot (\log_p(v_{\alpha \otimes \bar \beta}) u - \log_p(u_{\alpha \otimes \bar \beta}) v).
\end{equation}
\end{lemma}
\begin{proof}
This follows by applying the same result to the $p$-adic $L$-value associated to $g_{1/\beta}^*$, which agrees with the former since Dasgupta's factorization allows us to say that $L_p^{g_{\alpha}}(g_{\alpha},g_{1/\beta}^*,1) = L_p^{g_{1/\beta}^*}(g_{1/\beta}^*,g_{\alpha},1)$ (alternatively, this follows after a direct computation with Hida's factorization, see Proposition \ref{impad}). Here, the latter corresponds to the function obtained by interpolating along the region where the Hida family attached to $g_{1/\beta}^*$ is dominant. Further, the product of the periods arising when pairing with the differentials is a rational quantity, as discussed in \cite[Section 5.2]{RiRo1}, so it does not affect the result.
\end{proof}

Therefore, we may determine the derivative along the direction cyclotomic direction (keeping the weights fixed) by a linear algebra argument.

\begin{propo}\label{deri-s}
Assume that the derivative of $\kappa_p(\hg,\hg^*)$ along the cyclotomic derivative is non-vanishing. Then, up to multiplication by scalar, the cyclotomic derivative of the Beilinson--Flach class $\kappa(\hg,\hg^*)$ at $(y_0,y_0,0)$ is \[ \frac{\partial \kappa_p(\hg,\hg^*)}{\partial s} \Big|_{(y_0,y_0,0)} =\Omega \cdot (\log_p(v_{\alpha \otimes \bar \beta}) u - \log_p(u_{\alpha \otimes \bar \beta}) v) \pmod{L^{\times}}, \] where $\Omega \in H_p$ is the period of equation \eqref{wt-y} and the equality holds in $H^1(\mathbb Q_p, V_{gg^*}(1))$.
\end{propo}
\begin{proof}
Recall that the class vanishes along the line $(\ell,\ell,\ell-1)$. Hence, the result follows from equations \eqref{wt-y} and \eqref{wt-z} combined with the fact that \[ (0,0,1) = (1,1,1)-(1,0,0)-(0,1,0). \]
\end{proof}

We will see in the next section that the vanishing of the cyclotomic derivative is equivalent to the vanishing of the special value $L_p^{g_{\alpha}}(g,g^*,1)$ and also of the regulator corresponding to the $\mathcal L$-invariant. Observe that the previous results show that the different derived classes, which are elements living in a two-dimensional space, span the same line!

We can now mimic the approach of \cite{RiRo1} when proving the derived reciprocity law and obtain an expression for $L_p^{g_{\alpha}}(g,g^*,1)$ involving the period $\Omega$. In particular, considering the derivative along the analytic direction $(1,0,0)$ we have the following.
\begin{propo}\label{step-b}
Up to multiplication by an element in $L^{\times}$, the following equality holds \[ L_p^{g_{\alpha}}(g,g^*,1) \cdot \Big( \frac{-\alpha_{\hg}'}{\alpha_g} \Big)_{|y_0} = \Omega \cdot   (\log_p(u_{\alpha \otimes \bar \beta}) \cdot \log_p(v_1) - \log_p(v_{\alpha \otimes \bar \beta}) \cdot \log_p(u_1)). \]
As usual $\alpha_{\hg}$ stands for the derivative of the $U_p$-eigenvalue along the weight direction.
\end{propo}
\begin{proof}
This follows from making explicit the Euler factors in the explicit reciprocity law of \cite[\S10]{KLZ} and taking derivatives along the $y$-direction.
\end{proof}

In the next section, our aim is determining the value of the period $\Omega$ appearing in Proposition \ref{deri-s}, which would complete the proof of our main theorem.

\section{Cyclotomic derivatives and proof of the main theorem}

\subsection{Cyclotomic derivatives}

Along this section, we assume that $g$ and $g^*$ do not move along Hida families and we just consider the cyclotomic variation. As an abuse of notation, write $\kappa(g,g^*,s) := \kappa(\hg,\hg^*)(y_0,y_0,s)$ to emphasize the dependence on $s$. The image of this class under the Perrin-Riou map recovers the $p$-adic $L$-function $L_p^{g_{\alpha}}(g,g^*,s)$, that is,
\begin{equation}
\langle \mathcal L_{gg^*}^{-+}(\kappa_p^{-+}(g,g^*,s)), \eta_g \otimes \omega_{g^*} \rangle = L_p^{g_{\alpha}}(g,g^*,s).
\end{equation}
Although we have shown that $\kappa_p(g,g^*,0)$ is zero, we do not expect that $L_p^{g_{\alpha}}(g,g^*,0)=0$ in general. This is the same situation we previously found in the setting of circular units: the Kubota--Leopoldt $p$-adic $L$-function of a non-trivial, even Dirichlet character $L_p(\chi,s)$ is seen as the image of a $\Lambda$-adic cohomology class $\kappa(\chi,s)$ under a Perrin-Riou map; unfortunately, it happens that $\kappa(\chi,1)=0$ when $\chi(p)=1$ and an Euler factor also vanishes, so we cannot assert (and indeed it is false!) that $L_p(\chi,1)=0$.

Along this section, and since there is no possible confusion, we write $\kappa'(g,g^*,s)$ for the cyclotomic derivative. Define the improved Perrin-Riou map as
\begin{equation}
\langle \widetilde{\mathcal L_{gg^*}^{-+}}, \eta_g \otimes \omega_{g^*} \rangle = \frac{\gamma-1}{\frac{1}{p} \log_p(\gamma)} \times \langle \mathcal L_{gg^*}^{-+}, \eta_g \otimes \omega_{g^*} \rangle \, : \, H^1(\mathbb Q_p, V_{gg^*}^{-+}(1-s)) \longrightarrow \Lambda.
\end{equation}
Therefore, we have
\begin{equation}
\langle \widetilde{\mathcal L_{gg^*}^{-+}}(\kappa_p'^{-+}(g,g^*,s)), \eta_g \otimes \omega_{g^*} \rangle = p \cdot L_p^{g_{\alpha}}(g,g^*,s).
\end{equation}
Hence, the value of $\langle \mathcal L_{gg^*}^{-+}(\kappa_p^{-+}(g,g^*,s)), \eta_g \otimes \omega_{g^*} \rangle$ agrees with
\begin{equation}\label{esto}
\langle \widetilde{\mathcal L_{gg^*}^{-+}}(\kappa_p'^{-+}(g,g^*,s)),\eta_g \otimes \omega_{g^*} \rangle.
\end{equation}

For the following result, consider as usual the identification
\begin{equation}
H^1(\mathbb Q_p, \ad^0(V_g)(1)) \simeq H_p^{\times}[\ad^0(g)] \otimes L_p,
\end{equation}
and take the element $\kappa_p'(g,g^*,0)$, which belongs to the latter space (and may be therefore identified with a local unit in $H_p^{\times}$). The same study of \cite[Remark 6.5]{Buy1} works verbatim in our setting, where he argues that the improved Perrin-Riou map is a multiple of the order map applied to the derived class. However, we want to find out this explicit multiple (at least, up to multiplication by a rational constant). Compare for example this setting with the computations of \cite[Proposition 2.5.5]{LVZ} and the discussions of Section 3 of loc.\,cit., showing that the improved exponential map they consider is indeed the order map (up to sign).

\begin{propo}\label{step-c}
Identifying $\kappa_p'^{-+}(g,g^*,0)$ with an element in $(H_p^{\times} \otimes L)^{G_{\mathbb Q_p}}$, one has \[ L_p^{g_{\alpha}}(g,g^*,0) = \ord_p(\kappa_p'^{-+}(g,g^*,0)) \pmod{L^{\times}}. \]
\end{propo}
\begin{proof}

We can rephrase the statement in terms of the well-known theory of Coleman's power series. Then, $\kappa_p^{-+}(g,g^*,s)$ may be seen as a compatible system of units varying over the cyclotomic $p$-tower, but whose bottom layer is trivial. Hence, we may use the properties of universal norms and Coleman maps, and invoke the results developed in the proof of \cite[Proposition 3.6]{Ven}, and more precisely (via duality) equation (27).

Then, the $p$-adic $L$-value can be obtained applying to $\kappa_p'(g,g^*,0)$ the map \[ \ord^{-+} \, : \, H^1(\mathbb Q, V_{gg^*} \otimes L_p(1)) \xrightarrow{\pr^{-+}} H^1(\mathbb Q_p, L_p(1)) \xrightarrow{\widetilde{\mathcal L_{gg^*,0}^{-+}}} \mathbb Q_p, \] where arguing as in \cite[Section 5.2]{RiRo1}, the map $\widetilde{\mathcal L_{gg^*,0}^{-+}}$ corresponds to the usual $p$-adic order map multiplied by the $p$-adic period $\Omega_{g_{\alpha}} \Xi_{g_{1/\beta}^*}$, which belongs to $L^{\times}$.

According to \eqref{esto}, we conclude that \[ L_p(g,g^*,0) = \ord_p(\kappa_p'^{-+}(g,g^*,0)) \pmod{L^{\times}}, \] as desired.
\end{proof}

Roughly speaking, the previous theorem says that the derivative of the logarithm is the order (which can be seen as the dual of the result which interprets the derivative of the dual exponential as a logarithm).


\begin{coro}
With the notations introduced along the previous section, and up to multiplication by $L^{\times}$, \[ L_p^{g_{\alpha}}(g,g^*,1) = \Omega \cdot \log_p(u_{\alpha \otimes \bar \beta}). \]
\end{coro}
\begin{proof}
This directly follows by combining Propositions \ref{step-c} and \ref{deri-s}.
\end{proof}

This can be connected again with the case of circular units, that is, $L_p(g,g^*,s)$ is also the order of the derivative of $\kappa(g,g^*,s)$ along the $s$-direction.

As usual, let \begin{equation}\label{citro}
\mathcal L(\ad^0(g_{\alpha})) :=  \frac{-\alpha'_{\hg}(y_0)}{\alpha_{\hg}(y_0)},
\end{equation}
where recall $\alpha_{\hg} = a_p(\hg) \in \Lambda_{\hg}$ is the Iwasawa function given by the eigenvalue of the Hecke operator $U_p$ acting on $\hg$, and $\alpha'_{\hg}$ is its derivative.

\begin{propo}
Assume that the $\mathcal L$-invariant $\mathcal L(\ad^0(g_{\alpha}))$ is non-zero. Then, it may be written as
\[ \mathcal L(\ad^0(g_{\alpha})) = \frac{\log_p(u_{\alpha \otimes \bar \beta}) \cdot \log_p(v_1) - \log_p(v_{\alpha \otimes \bar \beta}) \cdot \log_p(u_1)}{\log_p(u_{\alpha \otimes \bar \beta})} \pmod{L^{\times}}. \]
\end{propo}
\begin{proof}
Combining Proposition \ref{step-b} with Proposition \ref{step-c}, we have that \[ \frac{\Omega}{\mathcal L(\ad^0(g_{\alpha}))} \cdot \Big( \frac{\log_p(u_{\alpha \otimes \bar \beta}) \cdot \log_p(v_1) - \log_p(v_{\alpha \otimes \bar \beta})}{\log_p(u_{\alpha \otimes \bar \beta})} \Big) = \Omega \pmod{L^{\times}}. \] Dividing by $\Omega$ (provided that this value is non zero), the result follows.
\end{proof}

\subsection{Improved p-adic L-functions}

We finish the proof with the same argument invoked in \cite{RiRo1}, involving Hida's improved $p$-adic $L$-function. Then, the main theorem is automatically proved by virtue of the following result, which is Proposition 2.5 of loc.\,cit.
\begin{propo}\label{impad}
For a crystalline classical point $y_0\in \cW_{\hg}^\circ$ of weight $\ell \geq 1$, we have
\[  \mathcal L(\ad^0(g_{\alpha})) = L_p({\bf g},{\bf g^*})(y_0,y_0,\ell) = L_p'(\ad^0(g_{y_0}),\ell), \]
up to a non-zero rational constant. Then, Theorem \ref{central} in the introduction holds.
\end{propo}

\section{A reinterpretation of the special value formula}

The result we have proven along the last two sections was presented in the introduction as a special value formula for the Hida--Rankin $p$-adic $L$-function. Alternatively, it can be regarded as the computation of the $\mathcal L$-invariant for the adjoint of a weight one modular form, and following the original formulation given by Darmon, Lauder, and Rotger, it also admits a reinterpretation in terms of $p$-adic iterated integral (and this point of view is specially useful towards computational experiments). In order to give a more conceptual view of our results, and how they fit in the theory of exceptional zeros and Galois deformations of modular forms, we would like two discuss two other interpretations which were already behind the scenes in our joint works with Rotger.

\subsection{Deformations of weight one modular forms}

As usual, fix a $p$-stabilization $g_{\alpha}$ of the weight one modular form $g \in S_1(N,\chi)$. We discuss a reinterpretation of the main results in terms of deformations of modular forms, in a striking analogy with the different works around the Gross--Stark conjecture, and which may be useful towards generalizations of the main results to totally real fields, following the recent approach of Dasgupta, Kakde, and Ventullo.

Let $E_k$ denote the weight $k$ Eisenstein series, whose Fourier expansion is given by
\begin{equation}\label{eis1}
E_k = \frac{\zeta(1-k)}{2} + \sum_{n=1}^{\infty} \sigma_{k-1}(n)q^n, \quad \text{ where } \sigma_{k-1}(n) = \sum_{d \, | \, n} d^{k-1}.
\end{equation}
There are two possible ways of considering its $p$-adic variation in families, either by taking the ordinary $p$-stabilization, $E_k^{\ord}$, or the critical one, $E_k^{\crit}$. For the sake of simplicity, we restrict to the ordinary $p$-stabilization, and after further normalizing by $\zeta(1-k)/2$, we have the usual Eisenstein series $G_k^{(p)}$, given by \[ G_k^{(p)} = 1+ 2\zeta_p(1-k)^{-1} \sum_{n=1}^{\infty} \sigma_{k-1}^{(p)}(n)q^n. \] Since $\zeta_p(1-k)$ has a pole at $k=0$, it turns out that we can formally write $G_0^{(p)}=1$, and it makes sense to consider its derivative \[ (G_0^{(p)})' := 2(1-p^{-1})^{-1} \cdot \sum_{n=1}^{\infty} \sigma_{-1}^{(p)} q^n. \]
Further, it is possible to take the infinitesimal deformation $G_0^{(p)}+\varepsilon(G_0^{(p)})'$, and then multiplying by $g_{\alpha}$ we obtain
\begin{equation}\label{def1}
(G_0^{(p)}+ \varepsilon (G_0^{(p)})') \cdot g_{\alpha} = g_{\alpha} + \varepsilon (G_0^{(p)})' g_{\alpha}.
\end{equation}
We regard this expression as a modular form of weight $1+\varepsilon$ corresponding to an infinitesimal deformation of $g_{\alpha}$.

There is another natural deformation of $g_{\alpha}$ we want to consider, which is precisely the one behind the scenes in \cite{RiRo1} and which also appeared in \cite{DLR2}. This is defined as \[ g_{\alpha}' := \Big( \frac{d}{dy} \hg_{\alpha} \Big)_{|y=y_0}. \] Then, we may take a second deformation of the modular form $g_{\alpha}$, given by
\begin{equation}\label{def2}
g_{\alpha}+\varepsilon g_{\alpha}'.
\end{equation}

Let $e_{\ord}$ stand for the ordinary projector, and $e_{g_{\alpha}}$ for the projector onto the $g_{\alpha}$-isotypic component.
\begin{propo}
Under the running assumptions, \[ e_{g_{\alpha}} e_{\ord} (g_{\alpha} E_0^{[p]}) = (1-\alpha_g U_p^{-1})g_{\alpha}' \pmod{L^{\times}}. \]
\end{propo}
\begin{proof}
Substracting the deformations in \eqref{def1} and \eqref{def2}, we obtain $g_{\alpha} (G_0^{(p)})'-g_{\alpha}'$.
If we furthermore take the ordinary projection and project to the $g_{\alpha}$-component, we obtain a multiple of $g_{\alpha}$, that is
\begin{equation}
e_{g_{\alpha}} e_{\ord} (g_{\alpha} (G_0^{(p)})') = g_{\alpha}' + \mathcal L \cdot g_{\alpha}.
\end{equation}
Next, if we apply the operator $1-\alpha_g U_p^{-1}$ to both sides of the previous equation, the left hand side becomes just the $p$-depletion
\begin{equation}
e_{g_{\alpha}} e_{\ord} (g_{\alpha} (G_0^{[p]})'),
\end{equation}
while in the right hand side the operator $1-\alpha_g U_p^{-1}$ annihilates $g_{\alpha}$. We have thus proved the result.
\end{proof}

Note that the left hand side is an explicit multiple of the $p$-adic $L$-function $L_p^{g_{\alpha}}(g,g^*,1)$, so the proposition asserts that the $\mathcal L$-invariant which governs the arithmetic of the adjoint may be read as a generalized eigenvalue attached to the deformation $g_{\alpha}'$, that is,
\begin{equation}
(1-\alpha_g U_p^{-1})g_{\alpha}' = \mathcal L(\ad^0(g_{\alpha})) \cdot g_{\alpha} \pmod{L^{\times}}.
\end{equation}

\subsection{Spaces of generalized eigenvectors}

This section is merely speculative and tries to shed some light into very natural questions around the previously discussed phenomena. When discussing circular units, we have seen that the condition $\chi(p)=1$ provides us with a $p$-unit in an extended Selmer group, and we have discussed how to mimic this approach in the case of Beilinson--Flach elements.
But more generally, given two modular forms $g$ and $h$ of weights $\ell$ and $m$, respectively, and an integer $s$, there is a geometric construction of the so-called Eisenstein classes $\Eis^{[g,h,s]}$ whenever the triple $(\ell,m,s)$ satisfies the {\it weight condition} of \cite[Section 7]{KLZ}, that is, \[ 1 \leq s < \min\{\ell,m\}. \] This includes all the points of weights $(\ell,\ell,\ell-1)$ when $\ell \geq 2$, and in particular the classical weight two situation. This suggests an interpretation of the derived class as a limit of Eisenstein classes (which are generically non-vanishing) for weights $\ell \geq 2$.

The proof of the explicit reciprocity law for Beilinson--Flach classes rests on an explicit connection between the $p$-depleted Eisenstein series (which encodes values of the $p$-adic $L$-function) and the $p$-stabilized one (which encodes values of the regulator of a geometric cycle). In this note we recovered the expressions for the logarithm of the derived class in terms of $p$-adic $L$-values, but it is natural to look for a reciprocity law involving $\Eis^{[g,h,s]}$, whenever $h=g^*$ and $s=\ell-1$. Note that this was the only case excluded by \cite[Theorem 6.5.9]{KLZ1}. Let us discuss the limitations for a result like that and that one may find natural in this framework.

According to \cite[Lemma 4.10]{DR1} (see also \cite[Lemma 6.5.8]{KLZ1}), one has
\begin{equation}
gE_0^{[p]} = (1-\alpha_g \cdot U_p^{-1}) \cdot (g_{\alpha} E_0^{(p)}),
\end{equation}
where $E_0^{[p]}$ (resp. $E_0^{(p)}$) stands for the $p$-depletion (resp. $p$-stabilization) of the weight 0 Eisenstein series $E_0$.
In the non self-dual case, the corresponding operator acting on the space of generalized eigenforms $S_1(Np)[[g_{\alpha}]]$ is invertible, and we obtain a straightforward linear relation.
But when $h = g^*$, the connection is more involved. In this case, consider a generalized eigenbasis $\{e_1, \ldots, e_n\}$ for the $U_p$-operator acting on the space of generalized (non-necessarily overconvergent) modular forms $S_1(Np)[[g_{\alpha}]]$, that is, \[ U_p \cdot e_1 = \alpha_g \cdot e_1, \quad U_p \cdot e_2 = e_1 + \alpha_g \cdot e_2, \quad \ldots, \quad U_p \cdot e_n = e_{n-1} + \alpha_g \cdot e_n. \] Hence, the matrix corresponding to the operator $1-\alpha_g \cdot U_p^{-1}$ acting on this space has the quantity $-1/\alpha_g$ all over the upper diagonal and zeros elsewhere. If we now apply this operator to $E_0^{(p)}g_{\alpha}$, written in this basis as $\sum \lambda_i e_i$, what we get in the first non-zero component is $-1/\alpha_g \cdot \lambda_2$. That is, the second vector of the generalized eigenbasis is the one which encodes the $p$-adic $L$-value. Therefore, one may consider two different classes.
\begin{enumerate}
\item[(a)] The class $\Eis^{[g,g^*,\ell-1]}$, where $\ell$ is the weight of $\ell$, is related with the first coefficient in the expansion in the generalized eigenbasis (see \cite[Corollary 6.5.7]{KLZ}). This controls the $p$-stabilization of the Eisenstein series.
\item[(b)] The derived class $\kappa'(g,g^*)$, constructed in \cite{RiRo1}, is related with the $p$-adic $L$-value, and hence with second coefficient in the generalized eigenbasis. This measures the $p$-depletion of the Eisenstein series.
\end{enumerate}

Hence, when $g \in S_{\ell}(N,\chi_g)$ is an ordinary modular form of weight $\ell \geq 2$, the two classes \[ \{\Eis^{[g,g^*,\ell-1]}, \kappa'(g,g^*)\} \] are a priori unrelated.


\begin{question}
Can we interpret the class $\Eis^{[g,g^*,\ell-1]}$ in terms of some $p$-adic $L$-value? And can we make sense of the limit of these classes for weights $(\ell,\ell,\ell-1)$ when $\ell$ goes to 1 $p$-adically?
\end{question}
Note that while the $p$-depleted class is connected with the usual $p$-adic $L$-value, a priori there is no natural $p$-adic avatar encoding the value of the $p$-stabilized class.

Observe that in the setting of diagonal cycles of \cite{BSV}, the authors take a different approach to the vanishing phenomenon, defining an {\it improved} class which is a putative geometric refinement to the analogue of the Eisenstein class, and which agrees up to some $\mathcal L$-invariant with the derived class.

\section{A conjectural $p$-adic $L$-function}

It is a somewhat vexing fact that our computations regarding the $\mathcal L$-invariant of the adjoint of a weight one modular form captures a $2 \times 2$ regulator encoding information about both a unit and a $p$-unit, while the most natural object to work would be the unit itself, as it occurs with the celebrated Gross--Stark conjecture. Similarly, one would expect to be able to construct an {\it Eisenstein $p$-adic $L$-function}, in such a way that appropriate special values of it also capture information about the Beilinson--Flach classes, in a way that we now make precise. This section is purely conjectural, and must be regarded as a failure in our current work, where we have not succeeded in studying these aspects.

Consider the most general setting in which $g$ and $h$ are two weight one modular forms. As we have already recalled, there are four Beilinson--Flach classes attached to the choice of $p$-stabilizations of $g$ and $h$,
\begin{equation}
\kappa(g_{\alpha},h_{\alpha}), \quad \kappa(g_{\alpha},h_{\beta}), \quad \kappa(g_{\beta},h_{\alpha}), \quad \kappa(g_{\beta},h_{\beta}).
\end{equation}
We know that different components of it are related to special values of $p$-adic $L$-functions. Take for instance the case of $\kappa(g_{\alpha},h_{\alpha})$. Considering its restriction to a decomposition group at $p$, we may take a decomposition of $\kappa_p(g_{\alpha},h_{\alpha})$ of the form
\begin{equation}
\kappa_p^{--}(g_{\alpha},h_{\alpha}) \otimes e_{\beta \beta}^{\vee} \oplus \kappa_p^{-+}(g_{\alpha},h_{\alpha}) \otimes e_{\beta \alpha}^{\vee} \oplus \kappa_p^{+-}(g_{\alpha},h_{\alpha}) \otimes e_{\alpha \beta}^{\vee} \oplus \kappa_p^{++}(g_{\alpha},h_{\alpha}) \otimes e_{\alpha \alpha}^{\vee},
\end{equation}
where $\{e_{\alpha \alpha}^{\vee}, e_{\alpha \beta}^{\vee}, e_{\beta \alpha}^{\vee}, e_{\beta \beta}^{\vee}\}$ is a basis for $V_{gh}^{\vee} = \Hom(V_{gh},L)$, where
$$
\Fr_p(e_{\alpha \alpha}^{\vee}) = \chi_{gh}^{-1}(p) \beta_g \beta_h \cdot e_{\alpha \alpha}^{\vee}\, , \, ... \,
,\,\Fr_p(e_{\beta \beta}^{\vee}) = \chi_{gh}^{-1}(p) \alpha_g \alpha_h \cdot e_{\beta \beta}^{\vee}.
$$

According to \cite[Proposition 8.2.6]{KLZ}, the component $\kappa_p^{--}(g_{\alpha},h_{\alpha})=0$ vanishes. In the same way, the components $\kappa_p^{-+}(g_{\alpha},h_{\alpha})$ and $\kappa_p^{+-}(g_{\alpha},h_{\alpha})$ are related to the special values of the Hida--Rankin $p$-adic $L$-functions $\Lp^{g_{\alpha}}$ and $\Lp^{h_{\alpha}}$, respectively.  Hence, it is natural to expect that the remaining component $\kappa_p^{++}(g_{\alpha},h_{\alpha})$ could arise as the special value of some $p$-adic $L$-function.

Following the analogy with the case of diagonal cycles and triple product $p$-adic $L$-functions, it would be attached to the triple $(E_2(1,\chi_{gh}^{-1}),g,h)$, but varying over the region where the Eisenstein family is dominant. Of course this is not possible, but let us work formally recasting to the theory of Perrin-Riou maps. In particular, we may consider the three-variable cohomology class $\kappa(\hg,\hh)$, take the restriction to the line where both $g$ and $h$ are fixed and take the image under the Perrin-Riou map. That way we would get an element over the Iwasawa algebra that we may denote $L_p^{\Eis}(g,h,s)$. It may be instructive to compare this with the scenario of triple products, where the existence of a third $p$-adic $L$-function $\Lp^f$, which at points of weight $(2,1,1)$ interpolates classical $L$-values, provides a richer framework and draws a more complete picture.

To simplify things, let us focus just on the case where both $g$ and $h$ are self dual, that is $h=g^*$. Recall that this situation naturally splits in two scenarios, namely $h_{\alpha}=g_{1/\beta}^*$ and $h_{\alpha}=g_{1/\alpha}^*$. As we have mentioned before, we expect the previous {\it cyclotomic} $p$-adic $L$-function to encode information about the logarithm of the unit $u$, and not about the apparently complicated regulator of our main result.

More concretely, assume that $\alpha_g \alpha_h = 1$, and take the class $\kappa(g_{\alpha},g_{1/\alpha}^*)$, although the same works verbatim for $\kappa(g_{\beta},g_{1/\beta}^*)$. From the general theory of Perrin Riou maps, we may consider the map
\begin{equation}
\langle \mathcal L_{gg^*}^{++}, \omega_g \otimes \omega_{g^*} \rangle \, : \, H^1(\mathbb Q_p, V_{gg^*}^{++}(\varepsilon_{\cyc} \underline{\varepsilon}_{\cyc}^{-1})) \longrightarrow I^{-1} \Lambda_{\hg},
\end{equation}
with specializations \[ \nu_s(\langle \mathcal L_{gg^*}^{++}, \omega_g \otimes \omega_{g^*} \rangle) \, : \, H^1(\mathbb Q_p, V_{gg^*}^{++}(1-s)) \longrightarrow \mathbb C_p, \] where \[ \nu_s(\langle \mathcal L_{gg^*}^{++}, \omega_g \otimes \omega_{g^*} \rangle) = \frac{1-p^{s-1}}{1-p^{-s}} \cdot \begin{cases} \langle \frac{(-1)^s}{(-s)!} \cdot \langle \log_{\BK}, \omega_g \otimes \omega_{g^*} \rangle & \text{ if } s<0 \\ (s-1)! \cdot \langle \exp_{\BK}^*, \omega_g \otimes \omega_{g^*} \rangle & \text{ if } s>1. \end{cases} \]
Observe that we have not said anything about the specializations at $s=0$ and at $s=1$.

When $s=0$ (resp. $s=1$), we are still in the region of the Bloch--Kato logarithm (resp. dual exponential map), but the expression $1-p^{-s}$ (resp. $1-p^{s-1}$) vanishes. Hence, and following \cite[Proposition 2.5.5]{LVZ} (see also the computations of \cite[Section 3.1]{Ven} and \cite[Section 6.3]{Buy1}), we have the following.

\begin{lemma}\label{ord0}
The map
\begin{equation} \nu_0(\langle \mathcal L_{gg^*}^{++}, \omega_g \otimes \omega_{g^*} \rangle) \, : \, H^1(\mathbb Q_p, V_{gg^*}^{++}(1)) \longrightarrow \mathbb C_p
\end{equation}
is given by
\begin{equation}\label{nu0}
\nu_0(\langle \mathcal L_{gg^*}^{++}, \omega_g \otimes \omega_{g^*} \rangle) = (1-p^{-1}) \cdot \langle \ord_p, \omega_g \otimes \omega_{g^*} \rangle.
\end{equation}
\end{lemma}

\begin{remark}
The situation is slightly different to that of circular units: there, the fact of taking the derived class was related to the fact that the Coleman map was connected to an {\it imprimitive} $p$-adic $L$-function, vanishing at the point of interest and whose derivative there corresponds to the special value of the Kubota--Leopoldt $p$-adic $L$-function.
\end{remark}

Define
\begin{equation}
L_p^{\Eis}(g,h,s) = \langle \mathcal L_{gg^*}^{++}(\kappa_p^{++}(g_{\alpha},g_{1/\alpha}^*,s)), \omega_g \otimes \omega_{g^*} \rangle,
\end{equation}
where $\kappa(g_{\alpha},g_{1/\alpha}^*,s)$ is the restriction of the 3-variable cohomology class to the cyclotomic line.

As a piece of notation for the following result, let $\mathcal L_{g_{\alpha}}$ stand for the period ratio introduced in \cite[Section 2]{DR2.5}.

\begin{propo}
The special value of the derivative of $L_p^{\Eis}(g_{\alpha},g_{1/\alpha}^*,0)$ satisfies that \[ L_p^{\Eis}(g_{\alpha},g_{1/\alpha}^*,0) = \frac{\mathcal L_{g_{\alpha}}}{\log_p(u_{\alpha \otimes \bar \beta})} \times \log_p(u_1) \pmod{L^{\times}}. \]
\end{propo}

\begin{proof}
When $s=0$, the denominator of the Perrin-Riou map $\mathcal L_{gg^*}^{++}$ vanishes and we are in the setting discussed before. Then, the Perrin-Riou map is given by the order followed by the pairing with the canonical differentials, as in \eqref{nu0}. Since according to the results of \cite[Section 4]{RiRo2} \[ \kappa(g_{\alpha},g_{1/\alpha}^*) = \frac{1}{\Xi_{g_{\alpha}} \Omega_{g_{1/\alpha}^*}}  \frac{\log_p(u_1) \cdot v - \log_p(v_1) \cdot u}{\log_p(u_{\alpha \otimes \bar \beta})} \pmod{L^{\times}}, \] the image of $\kappa_p(g_{\alpha},g_{1/\alpha}^*)$ under the map \eqref{nu0} agrees with \[ \log_p(u_1) \cdot \frac{\mathcal L_{g_{\alpha}}}{\log_p(u_{\alpha \otimes \bar \beta})} \pmod{L^{\times}}. \]
\end{proof}

Further, recall that according to \cite[Conjecture 2.3]{DR2.5}, we expect that $\mathcal L_{g_{\alpha}}$ must agree with $\log_p(u_{\alpha \otimes \bar \beta})$ up to multiplication by $L^{\times}$, and this would give just $\log_p(u_1)$ in the previous formula.

As a final comment, observe that the class $\kappa(g_{\alpha},g_{1/\beta}^*)$ vanishes, while neither the corresponding numerator nor the denominator of the Perrin-Riou map do. Hence, we expect the special value to be zero. However, it would be licit to take the derivative of both the class and the $p$-adic $L$-function.

\end{document}